\documentclass[10pt,leqno]{amsart}
\usepackage{amsmath,amssymb,latexsym,amsfonts}
\usepackage{amsthm}
\usepackage[all]{xy}
\numberwithin{equation}{section}
\numberwithin{equation}{section}	

\theoremstyle{plain}			

\newtheorem{theorem}{Theorem}[section]
\newtheorem{lemma}[theorem]{Lemma}
\newtheorem{prop}[theorem]{Proposition}
\newtheorem{corollary}[theorem]{Corollary}

\theoremstyle{definition}		

\newtheorem{definition}[theorem]{Definition}

\newtheorem{rmk}[theorem]{Remark}

\newcommand{\weak}{\rightharpoonup}

\newcommand{\N}{\mathbb{N}}

\newcommand{\R}{\mathbb{R}}
\newcommand{\C}{\mathbb{C}}
\newcommand{\Cc}{\mathcal{C}}
\newcommand{\Hc}{\mathcal{H}}
\newcommand{\St}{\widetilde{S}}

\newcommand{\Gt}{\widetilde{\Gamma}}
\newcommand{\Pm}{\mathcal{P}_m}

\newcommand{\capacity}{\text{cap}_{\omega,m}}
\newcommand{\Capacity}{\text{Cap}_{\omega,m}}

\newcommand{\f}{\varphi}
\newcommand{\ind}{1{\hskip -2.5 pt}\hbox{\textsc{I}}}
\newcommand{\Lc}{\mathcal{L}}
\newcommand{\Lct}{\mathcal{L}^*}
\begin{document}

\title{Solutions to degenerate complex Hessian equations}
\author{Lu Hoang Chinh}
\date{\today}
\maketitle
\begin{abstract}
Let $(X,\omega)$ be  an $n$-dimensional compact K\"{a}hler manifold. We study degenerate  complex Hessian equations  of the  form $(\omega+dd^c\varphi)^m\wedge \omega^{n-m}=F(x,\varphi)\omega^n.$  Under some natural conditions on  $F$, this equation has a unique continuous solution. When $(X,\omega)$ is rational homogeneous we further show that the solution is H\"{o}lder continuous.
\end{abstract}
\section{Introduction}
Let $(X,\omega)$ be a compact  K\"{a}hler manifold of complex dimension $n$. Fix an integer $m$   between $1$ and  $n$, and let $d,d^c$ denote the usual real differential operators $d:=\partial+\bar{\partial}, d^c=\frac{\sqrt{-1}}{2\pi}(\bar{\partial}-\partial)$ so that $dd^c =\frac{i}{\pi}\partial\bar{\partial}.$  We are studying  degenerate complex Hessian equations of the form 
\begin{equation}\label{eq: He}
(\omega+dd^c\varphi)^m\wedge \omega^{n-m}=F(x,\varphi)\omega^n,
\end{equation}
where the density $F: X\times \R \rightarrow \R^+$ satisfies some natural integrability conditions (see Theorem A below).

The case $m=1$ corresponds to the Laplace equation and the case $m=n$
corresponds to  degenerate complex Monge-Amp\`ere equations which have been studied intensively in recent years 
(see \cite{Bl03,Bl05,Bl12,BGZ08,BK07,EGZ09,GKZ08,GZ05,GZ07,Kol98,Kol02,Kol03,Kol05}). So, equation (\ref{eq: He}) is a generalization of both Laplace and Monge-Amp\`ere equations.  
 
The non degenerate complex Hessian equation on  compact K\"{a}hler manifold, where $F(x,\varphi)=f(x)$, with $0<f\in \Cc^{\infty}(X)$, has been studied recently in \cite{H09,HMW10,Jb10,DK12}. In \cite{H09} and \cite{Jb10}, the authors independently solved this equation with a strong  additional hypothesis, assuming $(X,\omega)$ has non negative  holomorphic bisectional curvature. Later on, in \cite{HMW10} an a priori $\Cc^2$ estimate was obtained without curvature assumption. Recently, using this estimate and a blowing up analysis suggested in \cite{HMW10}, Dinew and Kolodziej solved the equation  in full generality.

Following  Blocki \cite{Bl05} we develop a potential theory for the complex Hessian equation on compact K\"{a}hler manifold.  We define the class of $(\omega,m)$-subharmonic functions which is a generalization of the class of $\omega$-plurisubharmonic functions when $m=n.$ The definition of the complex Hessian operator on bounded $(\omega,m)$-subharmonic functions is delicate due to difficulties in  regularization process.

To go around this difficulty, we introduce a capacity and use it to define the concept of quasi-uniform convergence. This allows us to define a suitable class of bounded and quasi-continuous $(\omega,m)$-subharmonic functions on which the complex Hessian operator is well defined and continuous under quasi-uniform convergence. We show that this definition coincides with the definition in the spirit of Bedford and Taylor method for the complex Monge-Amp\`ere  operator. A comparison principle and convergence results for this operator are also established.  

With these potential tools in hand, we then consider the degenerate complex Hessian equation (\ref{eq: He}). The first main result of this paper is the following:
\medskip

\noindent{\bf Theorem A.} {\it Let $(X,\omega)$ be a $n$-dimensional compact K\"{a}hler 
manifold. Fix $1\leq m\leq n$.  Let $F: X\times \R \rightarrow [0,+\infty)$ be a function satisfying the following  conditions:

 (F1) for all $x\in X$,  $t\mapsto F(x,t)$ is non-decreasing and continuous,
 
(F2) for any fixed  $t\in \R$, there exists $p>n/m$ such that the function  $x\mapsto F(x,t)$ belongs to  $L^p(X)$,

(F3) there exists $t_0\in \R$ such that $\int_X F(.,t_0)\omega^n=\int_X \omega^n.$ 

\noindent Then there exists a function $\varphi\in \Pm(X,\omega)\cap \Cc^0(X)$ , unique up to an additive constant, such that		
$$
(\omega+dd^c\varphi)^m\wedge\omega^{n-m}=F(x,\varphi)\omega^n.
$$

\noindent Moreover if $\forall x\in X, \ t\mapsto F(x,t)$ is increasing,  then the solution is unique.}
\medskip

Note that the condition (F3) is automatically satisfied if $F(.,-\infty)=0$ and $F(.,+\infty)=+\infty.$ An important particular case is the exponential function $F(x,t)=f(x)e^t.$ 

A particular case of this result has been obtained in \cite{DK11}. The key point in their proof is a domination between volume and capacity. Our main result is proved using this technique and the recent result in the smooth case \cite{DK12}.

When $(X,\omega)$ is rational homogeneous with $\omega$ being invariant under the Lie group action, we can easily regularize $(\omega,m)$-subharmonic. Adapting the techniques in \cite{EGZ09} we obtain H\"{o}lder continuity of the solution:
\medskip

\noindent{\bf Theorem B.} {\it Under the same assumption  as in Theorem A, assume further that  $(X,\omega)$ is rational homogeneous and $\omega$ is invariant under the Lie group action. Then the unique solution is H\"{o}lder continuous with exponent $0<\gamma<\frac{2(mp-n)}{mnp+2mp-2n}$.}
\medskip

When $m=n$ we get the same exponent $\gamma$  as in \cite{EGZ09}.  
\bigskip
  
\noindent{\bf Acknowledgement.} The paper is part of my Ph.D Thesis. I would like to express my deep gratitude to my advisor, Professor Ahmed Zeriahi, for sacrificing his very valuable time for me. I wish to express my sincere gratitude to Professor Vincent Guedj for his very useful suggestions and discussions to improve the paper. I also wish to say a special word of thanks to Professor S\'ebastien Boucksom for his kind invitation to IMJ and useful discussions. This paper owes much to their help and constant encouragement.
\medskip

\section{Preliminaries} 
In this section we introduce the notion of $(\omega,m)$-subharmonic functions following Blocki's  ideas  \cite{Bl05} (see also \cite{DK11}).  Using classical techniques for plurisubharmonic functions we obtain similar results.  
\subsection{Elementary symmetric functions}
First, we recall some basic properties of elementary symmetric functions (see \cite{Bl05}, \cite{CW01}, \cite{Ga59}). We use the notations in \cite{Bl05}.
Let $S_k$, $k=1,...,n$ be the $k$-elementary symmetric function, that is, for 
$\lambda=(\lambda_1,...,\lambda_n)\in \R^n$,
\begin{equation*}
S_k(\lambda)= \sum_{1\leq i_1<i_2<...<i_k\leq
n}\lambda_{i_1}\lambda_{i_2}...\lambda_{i_k}.
\end{equation*}
Let $\Gamma_k$ denote the closure of the connected component of $\{
S_k(\lambda)>0\}$ containing $(1,...,1).$ It is easy to show that
$$
\Gamma_k=\big\{\lambda\in \R^n \ \mid\  S_k(\lambda_1+t,...,\lambda_n+t)\geq 0,  \ \forall t\geq 0\big\}.
$$
and hence
$$
\Gamma_k= \big\{\lambda\in \R^n\ \mid\ S_j(\lambda)\geq 0,\ \ \forall 1\leq j\leq k\big\}.
$$
We have an obvious inclusion $\Gamma_n\subset ...\subset \Gamma_1.$\\
By G{\aa}rding \cite{Ga59} the set $\Gamma_k$ is a convex cone in $\R^n$ and
$S_k^{1/k}$ is concave on $\Gamma_k.$ 
Let $\Hc$ denote the vector space (over $\R$) of complex hermitian $n\times n$
matrices. For $A\in \Hc$ we set
$$
\St_k(A)=S_k(\lambda(A)),
$$
where $\lambda(A)\in \R^n$ are is the vector of eigenvalues of $A.$ The function $\St_k$ can also be defined as the sum of all principal minors of order $k$,
$$
\St_k(A)=\sum_{\vert I\vert =k}A_{II}.
$$
From the latter we see that $\St_k$ is a homogeneous polynomial of order $k$ on $\Hc$ which is hyperbolic with respect to the identity matrix $I$ (that is for every $A\in
\St$ the equation $\St_k(A+tI)=0$ has $n$ real roots; see \cite{Ga59}). As in
\cite{Ga59} (see also \cite{Bl05}), the cone
$$
\Gt_k:=\big\{A\in \Hc\ \mid\ \St_k(A+tI)\geq 0, \forall t\geq 0\big\}=\{A\in \Hc\ \mid\ \lambda(A)\in \Gamma_k\}
$$
is convex and the function $\St_k^{1/k}$ is concave on $\Gt_k.$

\medskip


		


\subsection{$\omega$-subharmonic functions}
In this section, we consider $\Omega\subset X$  an open subset  contained in a local chart. 
\begin{definition}\label{def: weakly subharmonic}
A function $u\in L^1(\Omega)$ is called weakly $\omega$-subharmonic if 
$$
dd^cu\wedge \omega^{n-1}\geq 0,
$$
in the weak sense of currents.
\end{definition}
\noindent Thanks to Littman \cite{Lit63} we have the following approximation properties.
\begin{prop}\label{prop: Littman}
Let $u$ be a weakly $\omega$-subharmonic function in $\Omega$. Then there exists a one parameter family of functions $u_h$ with the following properties: for every  compact subset $\Omega'\subset \Omega$

a)  $u_h$ is smooth in $\Omega'$ for $h$ sufficiently large,

b)  $dd^cu_h\wedge \omega^{n-1}\geq 0$ in $\Omega',$

c) $u_h$ is non increasing with increasing $h,$ and $\lim_{h\to \infty}u_h(x)=u(x)$ almost every where  in $\Omega',$

d) $u_h$ is given explicitly as 
\begin{equation*}\label{eq: Littman}
u_h(y)=\int_{\Omega} K_h(x,y)u(x)dx,
\end{equation*}
where $K_h$ is a smooth non negative function and 
$$
\int_{\Omega} K_h(x,y)dy \to 1,
$$
uniformly in $x\in \Omega'.$ 
\end{prop}

\begin{definition}
A function $u$ is called $\omega$-subharmonic if it is weakly $\omega$-subharmonic  and for every $\Omega'\Subset \Omega$,  $\lim_{h\to\infty}u_h(x)=u(x), \forall x\in \Omega',$ where $u_h$ is constructed as in Proposition \ref{prop: Littman}. 
\end{definition}
\begin{rmk}
Any continuous weakly $\omega$-subharmonic function is $\omega$-subharmonic. \\
\indent If $(u_j)$ is a sequence of continuous $\omega$-subharmonic functions decreasing to $u$ and if $u\neq -\infty$ then $u$ is $\omega$-subharmonic. \\
\indent If $u$ is weakly $\omega$-subharmonic then the pointwise limit of $(u_h)$ is a $\omega$-subharmonic function.\\
\indent Let $(u_j)$ be a sequence of $\omega$-subharmonic functions and $(u_j)$ is uniformly bounded from above. Then $u:=(\limsup_j u_j)^{\star}$ is $\omega$-subharmonic. Where for a function $v$, $v^{\star}$ denotes the upper semicontinuous regularization of $v.$
\end{rmk}
\noindent The following Hartogs lemma can be proved in the same way as in the case of subharmonic functions.
\begin{lemma}\label{lem: Hartogs}
Let $u_t(x),t>0$ be a family of non positive $\omega$-subharmonic functions in $\Omega$ and $u_t$ is uniformly bounded in $L^1_{loc}(\Omega)$. Suppose that for compact subset $K$ in $\Omega$ there exists a constant $C$ such that $v(x)=[\limsup_{t\to +\infty}u_t(x)]^{\star}\leq C$ on $K.$ Then for every $\epsilon >0,$ there exists $T_{\epsilon}$ such that $u_t(x)\leq C+\epsilon$ for $t\geq T_{\epsilon}$ and $x\in K.$
\end{lemma} 

\subsection{$(\omega,m)$-subharmonic functions}
We associate real (1,1)-forms $\alpha$ in $\C^n$ with hermitian matrices
$[a_{j\bar{k}}]$ by
$$
\alpha=\frac{i}{\pi}\sum_{j,k}a_{j\bar{k}}dz_j\wedge d\bar{z_k}.
$$
Then the canonical K\"{a}hler form $\beta$ is associated with the identity
matrix $I.$ It is easy to see that 	
\begin{equation*}
\binom{n}{k}\alpha^k\wedge \beta^{n-k}=\St_k(A)\beta^n.
\end{equation*}

\begin{definition}
Let $\alpha$ be a real $(1,1)$-form on $X$. We say that $\alpha$ is
$(\omega,m)$-positive at a given point $P\in X$ if at this point we have
$$
\alpha^k\wedge\omega^{n-k}\geq 0, \ \ \forall k=1,...,m.
$$
We say that $\alpha$ is $(\omega,m)$-positive if it is $(\omega,m)$-positive at
any point of $X.$
\end{definition}
\begin{rmk}
Locally at $P\in X$ with local coordinates $z_1,...,z_n$, we have
$$
\alpha=\frac{i}{\pi}\sum_{j,k}\alpha_{j\bar{k}}dz_j\wedge d\bar{z_k},
$$
and 
$$
\omega=\frac{i}{\pi}\sum_{j,k}g_{j\bar{k}}dz_j\wedge d \bar{z_k}.
$$
Then $\alpha$ is $(\omega,m)$-positive at $P$ if and only if the eigenvalues
$\lambda(g^{-1}\alpha)=(\lambda_1,...,\lambda_n)$ of the matrix
$\alpha_{j\bar{k}}(P)$ with respect to the matrix $g_{j\bar{k}}(P)$ is in
$\Gamma_m$. These eigenvalues are independent of any choice of local
coordinates. 
\end{rmk}
\medskip

We can show easily the following result:
\begin{prop}\label{prop: equivalent definition}
Let $\alpha\in \Lambda^{1,1}(X)$ be a real (1,1)-form on $X.$ Then $\alpha$ is
$(\omega,m)$-positive if and only if 
$$
\alpha\wedge \beta_1\wedge...\wedge \beta_{m-1}\wedge \omega^{n-m}\geq 0,
$$
for all $(\omega,m)$-positive forms $\beta_1,...,\beta_{m-1}.$
\end{prop}
\begin{definition}
A current $T$ of bidegree $(p,p)$ is said to be $(\omega,m)$-positive if 
$$
\alpha_1\wedge...\wedge \alpha_{n-p}\wedge T\geq 0,
$$
for all smooth $(\omega,m)$-positive (1,1)-forms $\alpha_i.$
\end{definition}
\noindent Following Blocki \cite{Bl05} we can define $(\omega,m)$-subharmonicity for non-smooth functions.

\begin{definition}\label{def: m sub non smooth}
A function $\varphi: X\rightarrow \R \cup \{-\infty\}$ is called $(\omega,m)$-subharmonic if the following conditions hold\\
\indent (i)  In any local chart $\Omega,$ given $\rho$ the local potential of $\omega$ and set $u:=\rho+\varphi$, then $u$ is $\omega$-subharmonic,  \\
\indent  (ii) for every smooth $(\omega,m)$-positive forms
$\beta_1,...,\beta_{m-1}$ we have, in the weak sense of distributions,
\[(\omega+dd^c\varphi)\wedge \beta_1\wedge...\wedge \beta_{m-1}\wedge
\omega^{n-m}\geq 0.\] 
\end{definition}

Let  $SH_m(X,\omega)$ be the set of all $(\omega,m)$-subharmonic functions on $X.$ Observe  that, by definition, any $\varphi\in SH_m(X,\omega)$ is upper semicontinuous. 

The following properties of $(\omega,m)$-subharmonic functions are easy to show.

\begin{prop}\label{prop: properties of m sub functions}
(i) A $\Cc^2$ function $\f$ is $(\omega,m)$-subharmonic if and only if the form $(\omega+dd^c\f)$ is $(\omega,m)$-positive or equivalently 
$$
(\omega+dd^c\f)\wedge(\omega+dd^cu_1)\wedge...\wedge(\omega+dd^c u_{m-1})\wedge\omega^{n-m}\geq 0,
$$
for all $\Cc^2$ $(\omega,m)$-subharmonic functions $u_1,...,u_{m-1}.$

(ii) If $\varphi,\psi\in SH_m(X,\omega)$ then $\max (\varphi,\psi)\in SH_m(X,\omega).$

(ii) If $\varphi,\psi\in SH_m(X,\omega)$ and $\lambda\in [0,1]$ then $\lambda \varphi+(1-\lambda)\psi\in SH_m(X,\omega).$

(iii) If $(\varphi_j)\subset SH_m(X,\omega)$ is uniformly bounded from above then 
$(\limsup_{j}\varphi_j)^{\star}\in SH_m(X,\omega).$
\end{prop}
\noindent Thanks to Hartogs Lemma \ref{lem: Hartogs}, the following proposition can be proved in the same way as in the case of $\omega$-plurisubharmonic function (see \cite{GZ05}). 
\begin{prop}\label{prop: compactness property}
Let $(\varphi_j)$ be a sequence of functions in $SH_m(X,\omega)$. 

(i) If $(\varphi_j)$ is uniformly bounded from above on $X$, then either
$\varphi_j$ converges uniformly to $-\infty$ or the sequence $(\varphi_j)$ is
relatively compact in $L^1(X).$

(ii) If $\varphi_j\rightarrow \varphi$ in $L^1(X)$ then 
$$
sup_X \varphi= \lim_j sup_X \varphi_j.
$$
\end{prop}
\noindent The compactness result can be deduced easily from Proposition \ref{prop: compactness property}.
\begin{lemma}\label{lem: compactness}
There exists a constant $C_0>0$ such that for all $\varphi\in SH_m(X,\omega)$
satisfying $\sup_X\varphi=0$ we have
$$
\int_X \varphi \omega^n\geq -C_0.
$$
It then follows that 
$$
\Cc:= \{\varphi\in SH_m(X,\omega)\ \mid\ \sup_X\varphi\leq 0	; \int_X\varphi\omega^n\geq -C_0\}
$$
is a convex compact subset of $L^1(X).$
\end{lemma} 

\subsection{Non degenerate complex Hessian equations}
We summarize here some recent  results on the non degenerate complex Hessian equation on compact K\"{a}hler manifolds, 
\begin{equation}\label{eq: hes}
(\omega+dd^c\varphi)^m\wedge\omega^{n-m}=f\omega^n,
\end{equation}
where $0<f$ is smooth such that 
\begin{equation}\label{eq: compatibility condition}
\int_X f\omega^n=\int_X \omega^n.
\end{equation}

The following existence result was solved by Dinew and Kolodziej:
\begin{theorem}\cite{DK12}\label{thm: DK12}
 If $(X,\omega)$ is a compact K\"{a}hler manifold and  $0<f\in \Cc^{\infty}(X)$ satisfies (\ref{eq: compatibility condition}) then equation (\ref{eq: hes}) has a unique (up to an additive constant) smooth solution.
 \end{theorem}
 
 This result was known to hold when $(X,\omega)$ has non negative holomorphic bisectional curvature \cite{H09,Jb10}.
 
The complex Hessian equation in domains of $\C^n$, i.e. equations of the form 
\begin{equation*}
(dd^c u)^m\wedge \beta^{n-m}=f\beta^n,
\end{equation*}
where $\beta$ is the canonical K\"{a}hler form in $\C^n,$  was considered by Li \cite{Li04} and Blocki \cite{Bl05}.  Existence and uniqueness of smooth solution to the Dirichlet problem in smoothly bounded domains with $(m-1)$- pseudoconvex boundary  was obtained  in \cite{Li04}. In \cite{Bl05}, a potential theory for $m$-subharmonic functions  was  developed and the corresponding degenerate Dirichlet problem was solved. Recently,  Sadullaev and Abdullaev studied capacities and polar sets for $m$-subharmonic functions \cite{SA12}. Note that the corresponding problem when $\beta$ is not the euclidean K\"{a}hler form is fully open.

It is important to mention that the study of real Hessian equations is a classical subject which has been developed  previously in many papers, for example \cite{CNS85,CW01,ITW04,Kr95,La02,Tr95,TW99,Ur01,W09}.

 \section{Complex Hessian operators.}
  
 One of the key points in pluripotential theory is the smooth approximation which holds for quasi plurisubharmonic functions (\cite{BK07}, \cite{De92}). Such a result for $(\omega,m)$-subharmonic functions seems to be very difficult. To overcome this difficulty  we  work in an (a priori) restrictive class which is defined by means of uniform convergence with respect to capacity.  
\subsection{Capacity}   
\begin{definition}
Let $E\in X$ be a Borel subset. We define the inner $(\omega,m)$-capacity of $E$
by 
$$
\capacity (E):= \sup\Big\{\int_E \omega_{\varphi}^m\wedge\omega^{n-m}\ \mid\
\varphi\in SH_m(X,\omega)\cap \mathcal{C}^2(X), 0\leq \varphi\leq 1\Big\}.
$$
The outer $(\omega,m)$-capacity of $E$ is defined to be  
$$
\Capacity(E):=\inf\Big\{\capacity (U)\ \mid \ E\subset U,\ \ U\ \ \text{is an open
subset of X}\Big\}.
$$
\end{definition}
\noindent It follows directly from the definition that
 $\Capacity$ is monotone and $\sigma$-sub-additive.\\
Observe that if $\varphi\in SH_m(X,\omega)\cap \mathcal{C}^2(X)$, $0\leq \varphi\leq M$ then, for any Borel subset $E\subset X,$
\begin{equation}\label{eq: useful inequality}
\int_E\omega_{\varphi}^m\wedge \omega^{n-m}\leq M^m\capacity (E). 
\end{equation}

\begin{definition}\label{def: convergence in capacity}
A sequence $(\varphi_j)$ converges in $\capacity$ to $\varphi$ if for any $\delta>0$ we have
$$
\lim_{j\to\infty} \capacity (\vert\varphi_j-\varphi\vert>\delta)=0.
$$
\end{definition}
\begin{definition}\label{def: quasi-uniform convergence}
A sequence of functions $(\varphi_j)$ converges quasi-uniformly to $\varphi$ on
$X$ (w.r.t $\Capacity$) if for every $\epsilon>0$ there exists an open subset
$U\subset X$ such that $\Capacity (U)\leq \epsilon$ and $\varphi_j$ converges
uniformly to $\varphi$ in $X\setminus U.$
\end{definition}

\noindent This convergence is almost equivalent to the convergence in capacity as the following result shows
\begin{prop}\label{prop: equivalent convergence}
(i) If $\varphi_j$ converges quasi-uniformly to $\varphi$, then for each
$\delta>0$,
$$
\lim_{j\to\infty} \Capacity(\vert \varphi_j-\varphi\vert>\delta)=0.
$$
(ii) Conversely, assume that $(\varphi_j)$ is a sequence of functions and
$\varphi$ is a function such that, for every $\delta>0,$
$$
\lim_{j\to\infty} \Capacity (\vert\varphi_j-\varphi\vert>\delta)=0.
$$
Then there exists a subsequence $(\varphi_{j_k})$ converging quasi-uniformly to
$\varphi$. 
\end{prop}
\begin{proof}
The first part is obvious, so we only prove the second part. We can find a
subsequence (and for convenience we still denote it by $(\varphi_j)$) such that 
\[\Capacity (\vert\varphi_j-\varphi\vert>1/j)\leq 2^{-j},  \ \forall j.\]
For each $j$, let $U_j$ be an open subset of $X$ such that
$(\vert\varphi_j-\varphi\vert>1/j)\subset U_j$ and $\capacity (U_j)\leq
2^{-j+1}.$ 
Then for each $\epsilon>0$, we can find $k\in\N$ such that $\cup_{j\geq k}U_j$
is the open subset of $\Capacity$ less than $\epsilon$ and $\varphi_j$ converges
uniformly to $\varphi$ on its complement.
\end{proof}
\begin{definition}
We denote $\Pm(X,\omega)$ the set of all functions $\varphi\in SH_m(X,\omega)$
such that there exists a sequence of $\mathcal{C}^2$,
$(\omega,m)$-subharmonic functions $(\varphi_j)$  converging quasi-uniformly to
$\varphi$ on $X$. Equivalently, we can replace quasi-uniform convergence by
convergence in Capacity thanks to Proposition \ref{prop:  equivalent
convergence}.
\end{definition}

\begin{prop}
(i) Any $\varphi\in \Pm(X,\omega)$ is quasi continuous, that means, for any
$\epsilon>0$ there exists an open subset $U\subset X$ of $\Capacity$ less than
$\epsilon$ such that $\varphi$ is continuous on $X\setminus U$.\\
(ii) If $\varphi_j\downarrow \varphi$ in $\Pm(X,\omega)$ then $(\varphi_j)$
converges quasi-uniformly to $\varphi.$
\end{prop}
\begin{proof}
The first statement follows directly from the definition. From (i) , for each
$\epsilon>0$, there exists an open subset $U$ of $\capacity$ less than
$\epsilon$ such that $\varphi_j,\varphi$ are continuous on $X\setminus  U$ which
is compact. By Dini's Theorem, $\varphi_j$ converges uniformly to $\varphi$ on
$X\setminus U.$
\end{proof}
\noindent We have obvious inclusions
$$
SH_m(X,\omega)\cap \mathcal{C}^2(X)\subset \Pm(X,\omega)\subset SH_m(X,\omega),
$$
and
$$
PSH(X,\omega)\subset \Pm(X,\omega).
$$
\begin{rmk} 
Quasi-uniform convergence implies convergence point wise outside a subset of
$\Capacity$ zero. Moreover, if $\varphi_j$ is uniformly bounded and converges quasi-uniformly to
$\varphi$, then we have convergence in $L^p$ for every $p>0$. Indeed, for any
$\epsilon>0$ and an open subset $U$ as in definition \ref{def: quasi-uniform
convergence}, we have 
\begin{eqnarray*}
\int_X \vert \varphi_j-\varphi\vert^p \omega^n&\leq&  \int_{X\setminus U} \vert
\varphi_j-\varphi\vert^p\omega^n+\int_U \vert\varphi_j-\varphi\vert^p \omega^n
\\
&\leq& \int_{X\setminus U} \vert \varphi_j-\varphi\vert^p\omega^n+ \sup_{X,j}
\vert \varphi_j-\varphi\vert^p .\capacity (U)\\
&\leq& \int_{X\setminus U} \vert \varphi_j-\varphi\vert^p\omega^n+ C\epsilon.
\end{eqnarray*}
Taking the limsup over $j$ and then letting $\epsilon \to 0$ we obtain 
$$
\limsup_j \Vert \varphi_j-\varphi\Vert_p=0.
$$
\end{rmk}
\begin{lemma}\label{prop: properties of functions in class Pm}
If $\varphi,\psi$ belong to the class $\Pm(X,\omega)$ then so does $\max
(\varphi,\psi).$
\end{lemma}
 
 \begin{proof}
 Let $(\varphi_j), (\psi_j)$ be uniformly bounded sequences of functions in
$SH_m(X,\omega)\cap \mathcal{C}^2(X)$ converging quasi-uniformly to
$\varphi,\psi$ respectively. 
 Set 
 $$u:=\max(\varphi,\psi); \ u_j:= \max(\varphi_j,\psi_j); \ v_j:=
\frac{1}{j}\log(e^{j\varphi_j}+e^{j\psi_j}).$$
 For each $\epsilon>0$ there exists an open subset $U$ of $\capacity$ less than
$\epsilon$ and $\varphi_j,\psi_j$ converges uniformly on $X\setminus U$ to
$\varphi,\psi$ respectively. Since $u_j\leq v_j\leq \log(2)/j+u_j$ and $u_j$
converges uniformly to $u$ on $X\setminus U$ we deduce that $v_j$ converges
uniformly to $u$ on $X\setminus U.$
 \end{proof}

\subsection{Hessian measure}
In this section we define complex Hessian measure for functions in $SH_m(X,\omega)$ which can be approximated in $\Capacity$ by $\Cc^2$-functions in $SH_m(X,\omega)$. In particular, for functions in $\Pm(X,\omega)\cap L^{\infty}(X)$ this notion of Hessian measure can be defined by Bedford-Taylor's method.
\begin{theorem}\label{thm: hessian measure}
Let $\varphi\in SH_m(X,\omega)$ such that there exists a  uniformly bounded sequence $(\varphi_j)$ of $\Cc^2$ $(\omega,m)$-subharmonic functions converging in $\capacity$ to $\varphi.$ Then the sequence of measures $$H_m(\varphi_j):=
(\omega+dd^c\varphi_j)^m\wedge \omega^{n-m}$$ converges (weakly in the sense of measures) to a positive Radon measure $\mu.$ Moreover, the measure $\mu$ does not depend on the choice of the approximating sequence $(\varphi_j).$  We define the Hessian measure of $\varphi$ to be $H_m(\varphi):=\mu.$
\end{theorem}

\begin{proof}
Since all the measures $H_m(\varphi_j)$ have uniformly bounded mass (which is
$\int_X\omega^n$), they stay in a weakly compact subset. It suffices to show
that all accumulation points of this sequence are just the same. To do this it
is enough to show that for  every test function $\chi\in
\mathcal{C}^{\infty}(X),$ 
\[\lim_{j,k\to \infty}\int_X\chi [H_m(\varphi_j)-H_m(\varphi_k)]=0.\]
By integration by part formula we have
\begin{eqnarray}\label{eq: hessian measure 1}
\int_X\chi [H_m(\varphi_j)-H_m(\varphi_k)]&=&\int_X \chi
dd^c(\varphi_j-\varphi_k)\wedge T\\
&=&\int_X (\varphi_j-\varphi_k)dd^c\chi\wedge T,\nonumber
\end{eqnarray}
where
$$T=\sum_{l=0}^{m-1}(\omega+dd^c\varphi_j)^l\wedge(\omega+dd^c\varphi_k)^{m-1-l}
\wedge \omega^{n-m}.$$
Fix $\epsilon>0$, and set $U=U(j,k,\epsilon)=\{\vert \varphi_j-\varphi_k\vert \geq \epsilon\}.$ By $C$ we will denote a constant that does not depend on $j,k,\epsilon.$ Then by (\ref{eq: hessian measure 1}) and (\ref{eq: useful
inequality}) there exists $C>0$ such that
\begin{eqnarray*}\label{eq: hessian measure 2}
\Big\vert\int_X\chi [H_m(\varphi_j)-H_m(\varphi_k)]\Big\vert &\leq& \int_U
\vert\varphi_j-\varphi_k\vert C\omega\wedge T+\int_{X\setminus U}
\vert\varphi_j-\varphi_k\vert C\omega\wedge T\nonumber\\
&\leq&  C\capacity (U)+C\epsilon\sup_{X\setminus U} \vert\varphi_j-\varphi_k\vert\int_{X\setminus U}\omega\wedge T.\\
\end{eqnarray*}
Now, it follows that 
$$
\limsup_{j,k\to \infty}\Big\vert\int_X\chi [H_m(\varphi_j)-H_m(\varphi_k)]\Big\vert \leq C\epsilon.
$$
The result follows by letting $\epsilon\downarrow 0.$ For the independence in the choice of the sequence it is enough to repeat the above arguments. 
\end{proof}

\begin{lemma}\label{lem: capacity non smooth}
Let $U\subset X$ be an open subset and $\f$ be a bounded function in $\Pm(X,\omega).$ Then 
$$
\int_U H_m(\f)\leq 2(\sup_X \vert \f\vert+1)\Capacity (U).
$$
\end{lemma}
\begin{proof}
Let $\f_j$ be a sequence of $\Cc^2$ functions in $SH_m(X,\omega)$ converging quasi uniformly to $\f.$ We can assume that 
$$
-\sup_X \vert \f\vert-1 \leq \f_j\leq \sup_X \vert \f\vert +1, \ \forall j.
$$
Then 
$$
\int_U H_m(\f)\leq \liminf_{j\to+\infty}\int_U H_m(\f_j)\leq 2(\sup_X \vert \f\vert+1)\Capacity (U).
$$
\end{proof}
\noindent For functions in $\Pm(X,\omega)\cap L^{\infty}(X)$ we can also define the Hessian measure in a weak sense following Bedford-Taylor method. 
\begin{lemma}
Let $\varphi_1, \varphi_2\in \Pm(X,\omega)\cap L^{\infty}(X)$. Then the current
$\omega_{\varphi_1}\wedge \omega_{\varphi_2}\wedge \omega^{n-m}$ is well defined
in the weak sense (Bedford-Taylor), symmetric and $(\omega,m)$-positive. Then we can define inductively the Hessian measure of $\varphi\in \Pm(X,\omega)\cap L^{\infty}(X)$,
$$
H_m(\varphi):=(\omega+dd^c\varphi)^m\wedge \omega^{n-m}.
$$
Moreover, this definition coincides with the one in Theorem \ref{thm: hessian measure}.
\end{lemma}
\begin{proof}
It follows from definition of $(\omega,m)$-subharmonic functions that
$T_1=(\omega+dd^c\varphi_1)\wedge\omega^{n-m}$ is a $(\omega,m)$-positive
current. If $\varphi_2\in \Pm(X,\omega)\cap L^{\infty}(X)$ then $dd^c \varphi_2 \wedge T_1$ is the
current defined by 
$$
dd^c \varphi_2 \wedge T_1=dd^c (\varphi_2T_1).
$$
We denote by $T_2=\omega_{\varphi_1}\wedge \omega_{\varphi_2}\wedge
\omega^{n-m}.$ Since $\varphi_1, \varphi_2$ are in $\Pm(X,\omega)\cap L^{\infty}(X)$, there exist
uniformly bounded sequences $(\varphi_1^j), (\varphi_2^j)$ in $\Pm(X,\omega)\cap
\mathcal{C}^2(X)$  converging quasi-uniformly to $\varphi_1, \varphi_2$
respectively. The sequence of currents $T_2^j=\omega_{\varphi_1^j}\wedge
\omega_{\varphi_2^j}\wedge \omega^{n-m}$ converges to $T_2$ and hence $T_2$ is
$(\omega,m)$- positive and 
$$
\omega_{\varphi_1}\wedge \omega_{\varphi_2}\wedge
\omega^{n-m}=\omega_{\varphi_2}\wedge \omega_{\varphi_1}\wedge \omega^{n-m}.
$$
To prove that $T_2^j$ converges to $T_2$, let us choose some test form $\chi$
and prove the following convergence
\begin{equation}\label{eq: definition of Hessian measure in Pmt}
\lim_{j\to \infty}\int_X \chi\wedge dd^c(\varphi_2^j-\varphi_2)\wedge T_1=0.
\end{equation} 
We have
\begin{eqnarray*}
\Big\vert\int_X \chi\wedge dd^c(\varphi_2^j-\varphi_2)\wedge
T_1\Big\vert&=&\Big\vert\int_X (\varphi_2^j-\varphi_2)dd^c \chi \wedge
T_1\Big\vert\\
&\leq&C\int_X
\vert\varphi_2^j-\varphi_2\vert\omega_{\varphi_1}\wedge\omega^{n-1},
\end{eqnarray*}
where the constant $C$ depends only on $\chi,\omega.$ Now (\ref{eq: definition
of Hessian measure in Pmt}) follows from the last inequality in view of
$$
\int_U\omega_{\varphi_1}\wedge\omega^{n-1}\leq C\capacity (U),
$$
for every open subset $U\subset X.$ 
 \end{proof}
 
\noindent We can prove inductively that the current
$$
T_k=\omega_{\varphi_1}\wedge...\wedge \omega_{\varphi_k}\wedge \omega^{n-m}
$$
is well-defined, symmetric, $(\omega,m)$-positive, for each $k\leq m$ and
$\varphi_i\in \Pm(X,\omega)\cap L^{\infty}(X)$. The Hessian measure of $\varphi \in \Pm(X,\omega)\cap L^{\infty}(X)$
is defined in this way 
$$
H_m(\varphi)=\omega_{\varphi}\wedge...\wedge \omega_{\varphi}\wedge\omega^{n-m}.
$$
\noindent Now, given $\varphi\in \Pm(X,\omega)\cap L^{\infty}(X)$, it is easy to see that the Hessian measure of $\varphi$ defined by the above construction coincides with the
Hessian measure $H_m(\varphi)$ defined in Theorem  \ref{thm: hessian measure}.
 
\subsection{Some Convergence results}
In this section we state some convergence results and the comparison principle for functions in $\Pm(X,\omega)\cap L^{\infty}(X)$. The proofs are nearly the same as for the Monge-Amp\`ere operator and hence are omitted.
\begin{prop}\label{prop: convergence in capacity}
Let $(\varphi^1_j),...,(\varphi^m_j)$ be uniformly bounded sequence of functions
in $\Pm(X,\omega)\cap L^{\infty}(X)$ converging quasi-uniformly to $\varphi^1,...,\varphi^m$ respectively. Assume that $(f_j)$ is a uniformly bounded sequence of quasi continuous functions converging quasi uniformly to $\f.$ Then we have the weak convergence of measures  
\begin{equation*}\label{prop: convergence in capacity 1}
f_j\omega_{\varphi_j^1}\wedge...\wedge\omega_{\varphi_j^m}\wedge
\omega^{n-m}\rightharpoonup
f\omega_{\varphi^1}\wedge...\wedge\omega_{\varphi^m}\wedge \omega^{n-m}.
\end{equation*}
\end{prop}
\begin{proof}
Thanks to Lemma \ref{lem: capacity non smooth} we can follow the lines in \cite{Kol05}.
\end{proof}
\noindent The integration by parts formula is valid for $\mathcal{C}^2$ functions (by
Stokes). By Proposition \ref{prop: convergence in capacity} we see that it is also
valid for functions in $\Pm(X,\omega)\cap L^{\infty}(X).$
\begin{theorem}[Integration by parts]\label{thm: integration by parts}
Let $\varphi,\psi \in \Pm(X,\omega)\cap L^{\infty}(X)$ and $T$ be a current of the form 
$$
T=\omega_{\varphi_1}\wedge...\wedge \omega_{\varphi_{m-1}}\wedge\omega^{n-m},
$$
with $\varphi_i \in \Pm(X,\omega)\cap L^{\infty}(X).$ Then
$$
\int_X \varphi dd^c\psi \wedge T=\int_X \psi dd^c\varphi \wedge T.
$$
\end{theorem}
\noindent The maximum principle for functions in $\Pm(X,\omega)$ can be proved by the same way as in the classical case. 
\begin{theorem}[Maximum principle]\label{thm: comparison principle}
If $\varphi,\psi$ be two functions in $\Pm(X,\omega)\cap L^{\infty}(X)$ then 
$$
\ind_{\{\varphi>\psi\}}H_m(\max (\varphi,\psi))=\ind_{\{\varphi>\psi\}}H_m(\varphi).
$$
\end{theorem}

\noindent From Theorem \ref{thm: comparison principle} we easily get 
 \begin{corollary}[Comparison principle]\label{cor: comparison principle}
 If $\varphi,\psi\in \Pm(X,\omega)\cap L^{\infty}(X)$ then
$$
\int_{(\varphi>\psi)}H_m(\varphi)\leq \int_{(\varphi>\psi)}H_m(\psi). 
$$
\end{corollary}
 
 \begin{lemma}\label{lemma: capacity}
 Let $\varphi,\psi$ be two non positive  functions in $\Pm(X,\omega)\cap L^{\infty}(X).$ If $s>0$
and $0<t< 1$ then we have
 			\begin{equation}\label{eq: capacity}
			 t^m\Capacity(\varphi-\psi< -t-s)\leq (1+M)^m\int_{(\varphi-\psi< -s)}
			H_m(\varphi),
			 \end{equation}
 where $M=\Vert \psi\Vert_{L^{\infty}(X)}.$
 \end{lemma}
   \begin{proof}
We can assume that $\psi$ is continuous on $X$. For the general case we can approximate $\psi$ quasi-uniformly by sequence of $\Cc^2$  functions in $SH_m(X,\omega)$. In (\ref{eq: capacity}) We can replace $\Capacity$ by $\capacity$ since they coincide on open sets. Now, it suffices to repeat the arguments in \cite{EGZ09}.

   \end{proof}
   
   \begin{prop}[Chern-Levine-Nirenberg inequality]\label{prop:
Chern-Levine-Nirenberg inequality}
   Let $T$ be any current of the form $T=\omega_{u_1}\wedge
...\wedge\omega_{u_{m-1}}\wedge \omega^{n-m}$ with $u_1,...,u_{m-1}\in
\Pm(X,\omega)\cap L^{\infty}(X),$ and $\varphi, \psi$ be two functions in $\Pm(X,\omega)\cap L^{\infty}(X)$. Then
  \begin{equation}\label{ineq: CLN}
     \int_X \vert \psi\vert \omega_{\varphi}\wedge T\leq \int_X \vert\psi\vert
T\wedge\omega+ \Big(2\vert \sup_X\psi\vert
+\sup_X\varphi-\inf_X\varphi\Big)\int_X\omega^n.
   \end{equation}
   \end{prop}
    \begin{proof}
The proof is nearly the same as in \cite{GZ05} and is omitted.
\end{proof}
\noindent Applying (\ref{ineq: CLN}) for $T_i=\omega_{\varphi}^i\wedge \omega^{n-m+i}$
for $i=m-1,...,0$ we obtain 
\begin{corollary}\label{cor: Chern-Levine-Nirenberg inequality}
Let $\varphi, \psi$ be two functions in $\Pm(X,\omega)$ such that $0\leq
\varphi\leq 1$. Then
\begin{equation*}
\int_X \vert \psi\vert H_m(\varphi)\leq \int_X \vert\psi\vert\omega^n+
m\Big(2\vert \sup_X\psi\vert +1\Big)\int_X\omega^n.
\end{equation*}
\end{corollary}
\noindent Applying Corollary \ref{cor: Chern-Levine-Nirenberg inequality} we obtain:   
\begin{corollary}\label{cor: capacity sublevel set}
There exists a constant $C>0$ such that for all $\psi\in \Pm(X,\omega)$
satisfying $\sup_X\psi=-1$ and for every $t>0$ we have
$$
\Capacity (\psi<-t)\leq C/t.
$$  
\end{corollary}
We end this section by showing that the class $\Pm(X,\omega)$
is stable under decreasing sequences.
\begin{prop}
Let  $(\f_j)$ be a decreasing sequence of functions in  
$\Pm(X,\omega)$ converging to  $\f\not \equiv -\infty.$ 
Then $\f_j$ converges to $\f$ quasi-uniformly.
In particular, $\f\in \Pm(X,\omega).$ 
\end{prop}
\begin{proof}
It is easy to see that  $\f$ is $(\omega,m)$-subharmonic.  
It suffices to show that there exists a subsequence  
of  $(\f_j)$  converging quasi-uniformly  to $\f.$  \\
In view of Corollary \ref{cor: capacity sublevel set}
we can assume that $\f$ is bounded.\\
Fix $k\in \N.$ For each  $j>k\in \N,$ by applying
 Lemma \ref{lemma: capacity}  with $\f=\f_j, \psi=\f_k, s=t$ 
 we obtain
\begin{eqnarray}\label{eq: decrease 1}
 t^m\Capacity(\f_j-\f_k< -2t)&\leq &(1+M)^m\int_{(\f_j-\f_k< -t)}H_m(\f_j)\\
 &\leq &\frac{(1+M)^m}{t}\int_X (\f_k-\f_j)H_m(\f_j).\nonumber 
\end{eqnarray}
After extracting a subsequence if necessary
we can assume that $H_m(\f_j)\weak \mu$ in the 
weak sense of measures.We apply Lemma \ref{lem: decrease} below to get
$$
\lim_{j\to+\infty}\int_X \f_jH_m(\f_j)=\int_X\f d\mu.
$$
From the quasi-continuity of the functions $\f_j, j\in \N$ and
the  $\sigma$-subadditivity of $\Capacity$ we deduce that
for each fixed $\epsilon>0$ there exists an open subset $U$ 
such that  $\Capacity (U)<\epsilon$ and there exists a subsequence  
$(\tilde{\f}_j)$  of continuous functions on $X$ such that 
for any  $j,$ $\f_j=\tilde{\f}_j$ on $X\setminus U.$\\
From basic properties of $\Capacity$  we have
 \begin{eqnarray}\label{eq: decrease 2}
 t^m\Capacity (\f_j-\tilde{\f}_k< -2t)&\leq & t^m\Capacity (\f_j-\f_k< -2t)+ t^m\epsilon\\
 &\leq &\frac{(1+M)^m}{t}\int_X (\f_k-\f_j)H_m(\f_j)+t^m\epsilon.\nonumber
 \end{eqnarray}
 
Recall that  $\capacity$ is continuous under increasing sequence. 
Note also that $\Capacity$ and $\capacity$ coincide on open sets.
By taking the limit when  $j\to+\infty$ in (\ref{eq: decrease 2}), 
we obtain  
$$
t^m\Capacity (\f-\tilde{\f}_k< -2t)\leq 
\frac{(1+M)^m}{t}\int_X (\f_k-\f)d\mu+t^m\epsilon.
$$
It follows that 
$$
t^m\Capacity (\f-\f_k< -2t)\leq 
\frac{(1+M)^m}{t}\int_X (\f_k-\f)d\mu+2t^m\epsilon,
$$
and hence,
$$
\lim_{k\to+\infty} \Capacity(\f-\f_k< -2t)=0.
$$
Now, by Proposition \ref{prop: equivalent convergence} there 
exists a subsequence of $(\f_j)$ converging  quasi-uniformly to $\f$.
To complete the proof it remains to prove the following lemma. 
\end{proof}

\begin{lemma}\label{lem: decrease}
Assume that $(\f_j)$ is a sequence in  $\Pm(X,\omega)$  decreasing 
to  $\f\in L^{\infty}(X).$ If $H_m(\f_j)$ converges weakly  to $\mu$
in the sense of measures then 
 $$
 \lim_{j\to+\infty}\int_X \f_j H_m(\f_j)=\int_X \f d\mu.
 $$
 \end{lemma}
 \begin{proof}
 We prove this lemma by induction.  It obviously holds when $m=1.$ 
 Remark also that 
 $$
 \limsup_{j\to+\infty}\int_X \f_j H_m(\f_j)\leq \int_X \f d\mu.
 $$
 Thus, it suffices to prove that 
 $$
 \liminf_{j\to+\infty}\int_X \f_j H_m(\f_j)\geq \int_X \f d\mu.
 $$
 Fix $k\in \N$.  For each  $j>k$, By integration by parts we get
 \begin{eqnarray}\label{eq: decrease 3}
 \int_X  \f_k[H_m(\f_k)-H_m(\f_j)]&=&\int_X \f_k dd^c (\f_k-\f_j)\wedge T\wedge \omega^{n-m}\\
 &=&\int_X (\f_k-\f_j) dd^c \f_k \wedge T\wedge \omega^{n-m}\nonumber\\
 &\geq &-\int_X (\f_k-\f_j) T\wedge \omega^{n-m+1}\nonumber ,
 \end{eqnarray}
 where
 $$
 T=\sum_{p=0}^{m-1}(\omega+dd^c \f_k)^p\wedge(\omega+dd^c\f_j)^{m-1-p}.
$$ 
By setting $\psi_j=\frac{\f_j+\f_k}{2},$ we get
 $$
 T\wedge\omega^{n-m+1}\leq 2^{m-1}(\omega+dd^c\psi_j)^{m-1}\wedge \omega^{n-m+1}. 
 $$
As a consequence, (\ref{eq: decrease 3})  yields
\begin{eqnarray}\label{eq: decrease 4}
 \int_X  \f_k[H_m(\f_k)-H_m(\f_j)]&\geq &-2^m\int_X (\f_k-\psi_j) H_{m-1}(\psi_j)
 \end{eqnarray}
After extracting a subsequence if necessary, we can assume that
 $H_{m-1}(\psi_j)\weak \nu$ in the weak sense of measures.
 By  letting  $j\to+\infty$ in (\ref{eq: decrease 4}), the induction hypothesis gives us
 $$
 \int_X  \f_k[H_m(\f_k)-\mu]\geq -2^{m-1}\int_X (\f_k-\f) d\nu.
 $$
 We then infer that
  $$
 \liminf_{j\to+\infty}\int_X \f_j H_m(\f_j)\geq \int_X \f d\mu,
 $$
 and the result follows.
 \end{proof}
\section{Stability  results}
In this section we use the volume-capacity estimate in \cite{DK11} and mimic the arguments in \cite{EGZ09} to prove stability results for the complex Hessian equation. 

Using Blocki's technique  \cite{Bl03} we obtain the following stability results. 
\begin{theorem}\label{thm: stability Blocki}
Let $\varphi, \psi\in SH_m(X,\omega)\cap \Cc^2(X,\omega)$,  $r\geq 2$, and set $\rho=\varphi-\psi$. Then
			\[\int_X \vert \rho\vert^{r-2}d\rho \wedge d^c\rho\wedge\omega^{n-1}\leq C\Big(\int_X \vert \rho\vert^{r-2}
			 \rho(H_m(\psi)-H_m(\varphi))\Big)^{2^{1-m}},\]
where $C$ is a positive constant depending only on  $n,m,r$, and upper bounds of $\Vert \varphi\Vert_{L^{\infty}(X)}$, $\Vert \psi\Vert_{L^{\infty}(X)}$, and $\int_X\omega^n$.
\end{theorem}
\noindent From Theorem \ref{thm: stability Blocki} and Corollary \ref{def: convergence in capacity} we thus get 
\begin{corollary}\label{cor: Blocki stability}
Let $\varphi, \psi\in \Pm(X,\omega)\cap L^{\infty}(X)$,   and set $\rho=\varphi-\psi$. Then
$$
\int_X d\rho \wedge d^c\rho\wedge\omega^{n-1}\leq C\Big(\int_X \rho(H_m(\psi)-H_m(\varphi))\Big)^{2^{1-m}},
$$
where $C$ is a positive constant depending only on  $n, m$, and upper bounds of $\Vert \varphi\Vert_{L^{\infty}(X)}$, $\Vert \psi\Vert_{L^{\infty}(X)}$, and $\int_X\omega^n$.
\end{corollary}
\noindent Corollary \ref{cor: Blocki stability} is useful to prove uniqueness results as we will see in the proof of Theorem A. 
\begin{definition}\label{def: measures satisfy condition Q}
Let $\alpha>0, A>0.$ A Borel measure $\mu$ on $X$ satisfies condition
$\mathcal{Q}_m(\alpha,A,\omega)$ if for all Borel subsets $K$ of $X,$
$$
\mu(K)\leq A \Capacity(K)^{1+\alpha}.
$$
\end{definition}
 
\begin{prop}\label{prop: uniform estimate EGZ09}
Let $\mu$ be a Borel measure satisfying condition
$\mathcal{Q}_m(\alpha,A,\omega)$. Suppose that  $\varphi\in \Pm(X,\omega)$
solves $H_m(\varphi)=\mu$, and $\sup_X\varphi=-1.$ Then there exists a constant
$C=C(\alpha,A,\omega,n,m)$ such that
$$
\sup_X \vert \varphi\vert \leq C.
$$
\end{prop}
\noindent \textbf{Sketch of proof.}
Set 
$$
f(s):= [\Capacity(\varphi<-s)]^{1/m}.
$$
Observe that $f:\R^{+}\rightarrow \R^{+}$ is right continuous, decreasing with
$\lim_{+\infty}f=0.$ Since $\mu$ satisfies condition
$\mathcal{Q}_m(\alpha,A,\omega)$, it follows from  Lemma \ref{lemma: capacity}
applied to the function $\psi\equiv 0$ that $f$ satisfies the condition in Lemma
2.4 in \cite{EGZ09}. Moreover it follows from Corollary \ref{cor: capacity sublevel
set} that
$$
f(s)\leq C_1s^{-1/m},
$$
for some constant $C_1$ which only depends on $\omega.$ Thus, by following the lines
in \cite{EGZ09}, page 615, we have the desired uniform estimate. 
\begin{theorem}\label{thm: stability EGZ09}
Suppose that $\varphi,\psi\in \Pm(X,\omega)\cap L^{\infty}(X)$ satisfy
$$
\sup_X\varphi=\sup_X\psi=-1.
$$ 
Assume that $H_m(\varphi)$, $H_m(\psi)$ satisfy condition
$\mathcal{Q}_m(\alpha,A,\omega)$ for some $\alpha,A>0.$ Then there exists
$C=C(\alpha,A,\omega, \Vert \varphi \Vert_{L^{\infty}(X)}, \Vert \psi
\Vert_{L^{\infty}(X)})>0$ such that, for any $\epsilon>0,$
$$
\sup_X (\psi-\varphi)\leq \epsilon+C[\Capacity(\varphi-\psi<-\epsilon)]^{\alpha/m}.
$$
\end{theorem}
\begin{proof}
The same as in \cite{EGZ09}, Proposition 2.6.
\end{proof}
\noindent The following Proposition is due to Kolodziej and Dinew  \cite[Propsition 2.1]{DK11}. We include here a slightly different proof.
\begin{prop}\cite{DK11}\label{prop: DK11 1}
Let $1<p<\frac{n}{n-m}.$ There exists a constant $C=C(p,\omega)$ such that
for every Borel subset $K$ of $X$, we have
$$
V(K)\leq C\Capacity(K)^p,
$$
where $V(K):=\int_K\omega^n.$
\end{prop}
\begin{proof}
Fix an open subset $U$ such that $K\subset U.$ Solve the complex
Monge-Amp\`ere equation to find $u\geq 0$ such that $\omega_u^n=f\omega^n$ on
$X$, with $f=V(U)^{-1}\chi_U.$   From \cite{BGZ08}, Corollary 3.2, the solution
$u$ is continuous and moreover, for each $r>1$, 
$$
\sup_X u \leq C\Vert f\Vert_r^{1/n},
$$
where the constant $C=C(r,\omega)$ does not depend on $K.$ The inequality
between mixed complex Monge-Amp\`ere measures \cite{Di09} tells us that
$$
\omega_u^m\wedge \omega^{n-m}\geq f^{m/n}\omega^n.
$$
Thus since $u\in \Pm(X,\omega)\cap L^{\infty}(X)$, we obtain
$$
\Capacity(U)\geq  (\sup_Xu)^{-m}\int_UH_m(u)\geq  (\sup_Xu)^{-m}\int_Uf^{m/n}\omega^n
\geq  C^{-m} V(U)^{1-\frac{m}{rn}}.
$$
Thus, for every $r>1$, there exists a constant $C$ not depending on $K$ such
that $V(K)\leq C.\Capacity(K)^{\frac{nr}{nr-m}}$. The proof is complete.
\end{proof}
\noindent As a consequence of Proposition \ref{prop: DK11 1} we have some examples of
measures satisfying condition $\mathcal{Q}_m(\alpha,A,\omega).$
\begin{lemma}\label{lem: DK11 2}
Assume $\mu=f\omega^n$ is a Borel measure with $0\leq f\in L^p(X)$ for some
$p>n/m$. Then for any $0<\alpha<\frac{mp-n}{(n-m)p},$ there exists
$A_{\alpha}>0$ such that $\mu$ satisfies
$\mathcal{Q}_m(\alpha,A_{\alpha},\omega).$
\end{lemma}
\noindent The following stability theorem was established in \cite{EGZ09} for the Monge-Amp\`ere equation. 
\begin{theorem}\label{thm: Lr-Linfty}
Assume $H_m(\varphi)=f\omega^n, \ H_m(\psi)=g\omega^n,$ where $\varphi,\psi\in
\Pm(X,\omega)\cap \Cc^{0}(X)$ and $f,g\in L^p(X)$ with $p>n/m.$ Fix $r>0$. Then if $\gamma$ small enough such that $\frac{\gamma mq}{r-\gamma(r+mq)}<\frac{mp-n}{(n-m)p}$,
we have
$$
\Vert \varphi-\psi\Vert_{L^{\infty}(X)}\leq C\Vert	\varphi-\psi\Vert_{L^r(X)}^{\gamma},
$$
where $q=\frac{p}{p-1}$ denotes the conjugate exponent of $p,$ and the constant $C$
depends only on $n,m,p,r$ and upper bounds  of $\Vert f\Vert_p, \ \Vert g\Vert_p$.
\end{theorem}
\begin{proof}
Fix $\epsilon>0,$ and $\alpha>0$ to be chosen later. It follows from Theorem
\ref{thm: stability EGZ09} and Proposition \ref{prop: uniform estimate EGZ09} that 
$$
\Vert \varphi-\psi\Vert_{L^{\infty}(X)}\leq \epsilon + C_1[\Capacity(\vert \varphi-\psi\vert>\epsilon)]^{\alpha/m}.
$$
Applying Lemma \ref{lemma: capacity} we see that
$$
\Capacity(\vert \varphi-\psi\vert>\epsilon)\leq
\frac{C_2}{\epsilon^{m+r/q}}\int_X \vert \varphi-\psi\vert^{r/q}(f+g)\omega^n.
$$
It follows thus from H\"{o}lder's inequality that 
$$
\Capacity(\vert \varphi-\psi\vert>\epsilon)\leq \frac{C_3\Vert
f+g\Vert_p}{\epsilon^{m+r/q}}\Vert \varphi-\psi\Vert_{L^r}^{r/q}.
$$
Choose $\epsilon:= \Vert \varphi-\psi\Vert_{L^r}^{\gamma}.$ Then 
$$
\Capacity(\vert \varphi-\psi\vert>\epsilon)\leq C_4[\Vert
\varphi-\psi\Vert_{L^r}]^{r/q-\gamma(m+r/q)}.
$$
We infer that 
$$
\Vert \varphi-\psi\Vert_{L^{\infty}(X)}\leq \Vert
\varphi-\psi\Vert_{L^r(X)}^{\gamma}+C_5\Vert
\varphi-\psi\Vert_{L^r(X)}^{\gamma'},
$$
where $\gamma'=\frac{\alpha}{m}[r/q-\gamma(m+r/q)].$ We finally choose
$\alpha$ so that $\gamma = \gamma'$: this yields the desired estimate.
\end{proof}
  
\section{Proof of the main results}
\subsection{Proof of Theorem A}
We first prove the uniqueness. Suppose that $\varphi$ and $\psi$ are two continuous solutions of (\ref{eq: He}). Set $\rho:= \varphi-\psi.$ It follows from Corollary \ref{cor: Blocki stability} that
			\[\int_X d\rho\wedge d^c\rho\wedge\omega^{n-1}\leq C.
			\Big(\int_X \rho(H_m(\psi)-H_m(\varphi))\Big)^{2^{1-m}},\]
where $C$ is a positive constant.  Since $F$ is non decreasing in the second variable,  it follows from Stokes formula that
			\[0\leq \int_X \rho(H_m(\psi)-H_m(\varphi))=\int_X
			(\varphi-\psi)(F(.,\psi)-F(.,\varphi))\omega^n\leq 0.\]
Thus, 
			\[\int_X d\rho\wedge d^c\rho\wedge\omega^{n-1}=0,\]
which implies that $\rho$ is constant. If moreover $t\mapsto F(x,t)$ is increasing for every $x\in X$, it is easy to see that $\rho=0.$

To prove  the existence, we consider three cases. 
\medskip

\noindent {\bf Case 1}: $F$ does not depend on the second variable, $F(x,t)=f(x), \forall x,t$.\\
Take a sequence of smooth strictly positive functions $(f_j)$ converging to $f$ in $L^p(X).$ We can assume that $\int_Xf_j\omega^n=\int_X\omega^n$, for every $j$. We use the existence result in Theorem \ref{thm: DK12} to produce a sequence of smooth solutions $(\varphi_j)$ normalized by $\sup_X \varphi_j=0, \forall j.$ By passing to a subsequence we can assume that $(\varphi_j)$ converges in $L^1(X).$ Since $\Vert f_j\Vert_p$ is uniformly bounded, by Lemma \ref{lem: DK11 2} we can find $\alpha,A$ which do not depend on $j$ such that all the measures $f_j\omega^n$ satisfy condition $\mathcal{Q}_m(\alpha,A,\omega).$ By Proposition \ref{prop: uniform estimate EGZ09}, the sequence $(\varphi_j)$ is uniformly bounded. Now it follows from Theorem \ref{thm: Lr-Linfty} that  $\varphi_j$ converges uniformly to a continuous function $\varphi\in \Pm(X,\omega)$ which solves equation $H_m(\varphi)=f\omega^n$. 
 
In the next two cases we will use the Schauder fixed point Theorem.
\medskip

\noindent {\bf Case 2}: There exists $t_1\in\R$ such that $\int_X F(x,t_1)\omega^n>\int_XF(x,t_0)\omega^n$.\\
We set 
$$
\Cc:=\{\varphi\in SH_m(X,\omega)\ \mid\ 	\int_X \varphi\omega^n\geq -C_0;\ \sup_X \varphi\leq 0\},
$$
where $C_0$ is the constant introduced in Lemma \ref{lem: compactness}. It
follows that $\Cc$ is a compact convex subset of $L^1(X).$
\smallskip

Take $\psi\in \Cc,$ we use the result in case 1 to find
$\varphi\in \Pm(X,\omega)\cap \Cc^0(X)$ such that $\sup_X \varphi=0$ and
$$
H_m(\varphi)=F(.,\psi+c_{\psi})\omega^n,
$$
where $c_{\psi}\geq t_0$ is a constant such that 
\begin{equation}\label{eq: constant c xi}
\int_X F(.,\psi+c_{\psi})\omega^n=\int_X\omega^n.
\end{equation}
This can be done because $F$ satisfies conditions (F2) and (F3). Indeed, by Fatou's Lemma we have
$$
\liminf _{t\to+\infty}\int_XF(.,\psi+t)\omega^n \geq \int_XF(.,t_1)\omega^n>\int_X \omega^n.
$$
Moreover $\int_X F(.,\psi+t_0)\omega^n\leq \int_X F(.,t_0)=\int_X\omega^n$. Thus by continuity of $t\mapsto \int_X F(.,\psi+t)\omega^n$ we can find $c_{\psi}$ satisfying (\ref{eq: constant c xi}).  Observe that $\varphi$ is well-defined and does not depend on $c_{\psi}.$ Indeed, assume that $c_1, c_2$ are two constants such that 
$$
\int_X F(.,\psi+c_1)\omega^n=\int_X F(.,\psi+c_2)\omega^n=\int_X\omega^n,
$$
and $\varphi_1,\varphi_2$ are two continuous functions in $\Pm(X,\omega)$ such that
$$
H_m(\varphi_1)=F(.,\psi+c_1),\ \  H_m(\varphi_2)=F(.,\psi+c_2).
$$
Since $t\mapsto F(x,t)$ is  non decreasing for every $x\in X,$ we have $F(.,\psi+c_1)=F(.,\psi+c_2)$ almost every where on $X$. Thus by the uniqueness result above, $\varphi_1=\varphi_2+ c$ for some constant $c$ which must be $0$ by the normalization. Then we define the map $\Phi: \Cc\rightarrow \Cc, \ \psi \mapsto \varphi.$ 

Now we prove that $\Phi$ is continuous on $\Cc.$ Suppose that $(\psi_j)$ is a sequence in $\Cc$ converging to $\psi\in \Cc$ in $L^1(X)$ and let $\varphi_j=\Phi(\psi_j)$. We set $c_j:=c_{\psi_j}$ and prove that $(c_j)$ is uniformly bounded. Suppose in the contrary that $c_j\uparrow +\infty.$  By subtracting a subsequence if necessary we can assume that $\psi_j\to \psi$ almost everywhere in $X$. Then by Fatou's lemma we have
$$
\int_X\omega^n=\lim_{j\to+\infty}\int_X F(.,\psi_j+c_j)\omega^n\geq\int_X F(.,t_1)\omega^n,
$$
which is impossible. Therefore the sequence $(c_j)$ is  bounded. This implies that the sequence $(F(.,\psi_j+c_j))_j$ is bounded in $L^p(X),$ for some $p>n/m$ which does not depend on $j.$ To prove the continuity of $\Phi$ it suffices to show that any cluster point of $(\varphi_j)$ satisfies $\Phi(\psi)=\varphi.$ Suppose that  $\varphi_j\to \varphi$ in $L^1(X)$. It follows from Theorem \ref{thm: Lr-Linfty} that the sequence $(\varphi_j)$ is Cauchy in $\Cc^0(X)$. Thus $(\varphi_j)$ converges to $\varphi$ in $\Cc^0(X)$ and $\varphi\in \Pm(X,\omega)\cap \Cc^0(X)$. By subtracting a subsequence if necessary we can assume that $\psi_j\to \psi$ almost everywhere on $X$ and $c_j\to c.$ Since $t\mapsto F(x,t)$ is continuous we see that $F(.,\psi_j+c_j) \to F(.,\psi+c)$ almost everywhere. Thus $H_m(\varphi)=F(.,\psi+c)$ which means $\Phi(\psi)=\varphi$ and hence $\Phi$ is continuous on $\Cc.$

By the Schauder fixed point Theorem, it follows that $\Phi$ has a fixed point in $\Cc$, say $\varphi$. By definition of $\Phi$, the function $\varphi$ must be in the class $\Pm(X,\omega)\cap \Cc^0(X)$ and we have
$$
H_m(\varphi)=F(.,\varphi+c_{\varphi})\omega^n.
$$
The function $\varphi+c_{\varphi}$ is the required solution.
\medskip
   
\noindent {\bf Case 3}: $\int_X F(.,t)\omega^n=\int_XF(.,t_0)\omega^n, \forall t\geq t_0.$  In this case we have $F(x,t)=F(x,t_0)$ for all $t\geq t_0$ and for almost $x\in X.$ Thus, for every $t\geq t_0,$ 
$$
\Vert F(.,t_0)\Vert_{L^p(X)}=\Vert F(.,t)\Vert_{L^p(X)}.
$$
From Proposition \ref{prop: uniform estimate EGZ09} we can find a positive constant $C_1$ such that for any $\varphi\in \Pm(X,\omega)\cap \Cc^0(X)$ satisfying $\sup_X \varphi=0$ and
$$
H_m(\varphi)=f\omega^n,
$$
with $\Vert f\Vert_p\leq \Vert F(.,t_0)\Vert_p$ then 
$$
\varphi\geq -C_1.
$$
We set 
$$
\Cc':=\{\varphi\in SH_m(X,\omega)\ \mid \  -C_1\leq \varphi \leq 0\}.
$$
Then $\Cc'$ is a compact convex subset of $L^1(X).$
\smallskip

Take $\psi\in \Cc',$ we use the result in case 1 to find
$\varphi\in \Pm(X,\omega)\cap \Cc^0(X)$ such that $\sup_X \varphi=0$ and
$$
H_m(\varphi)=F(.,\psi+c_{\psi})\omega^n,
$$
where $t_0\leq c_{\psi}\leq t_0+C_1$ is a constant such that 
$$
\int_X F(.,\psi+c_{\psi})\omega^n=\int_X\omega^n.
$$
This can be done because $F$ satisfies the condition (F2) and (F3) . Indeed, 
$$
\int_X F(.,\psi+t_0)\omega^n\leq \int_X \omega^n\leq \int_X F(.,\psi+t_0+C_1)\omega^n. 
$$
Thus by continuity we can find $c_{\psi}$ as above. As in case 2, $\varphi$ is well-defined and does not depend on the choice of $c_{\psi}$. By the choice of $C_1$, we see that $\varphi\in \Cc'.$ So, we can define a map $\Phi: \Cc'\rightarrow \Cc'$ by setting $\Phi(\psi)=\varphi.$

\smallskip

Now we prove that $\Phi$ is continuous on $\Cc'.$ Suppose that $(\psi_j)$ is a sequence in $\Cc'$ converging to $\psi\in \Cc'$ in $L^1(X)$ and let $\varphi_j=\Phi(\psi_j)$. We set $c_j:=c_{\psi_j}$. For each $j\in\N$, 
$$
\int_X [F(.,\psi_j+c_j)]^p\omega^n \leq \int_X [F(.,c_j)]^p\omega^n=\int_X [F(.,t_0)]^p\omega^n.
$$ 
Therefore, the sequence $(F(.,\psi_j+c_j))_j$ is bounded in $L^p(X).$\\
As in case 2, we can assume that $\varphi_j\to \varphi$ in $L^1(X)$. It follows from Theorem \ref{thm: Lr-Linfty} that the sequence $(\varphi_j)$ is Cauchy in $\Cc^0(X)$. Thus $(\varphi_j)$ converges to $\varphi$ in $\Cc^0(X)$ and $\varphi\in \Pm(X,\omega)\cap \Cc^0(X)$. By subtracting a subsequence if necessary we can assume that $\psi_j\to \psi$ in $L^1(X)$ and $c_j\to c.$ Then $H_m(\varphi)=F(.,\psi+c)$ and $\Phi(\psi)=\varphi$ which implies that $\Psi$ is continuous on $\Cc'.$

By the Schauder fixed point Theorem, it follows that $\Phi$ has a fixed point in $\Cc'$, say $\varphi$. By definition of $\Phi$, the function $\varphi$ must be in the class $\Pm(X,\omega)\cap \Cc^0(X)$ and we have
$$
H_m(\varphi)=F(.,\varphi+c_{\varphi})\omega^n.
$$		
The function $\varphi+c_{\varphi}$ is the required solution.
\subsection{Proof of Theorem B}
In this section we consider a special class of compact K\"{a}hler manifolds. We assume that $(X,\omega)$ is a rational homogeneous manifold. That means $X=G/H$, where $G$ is a complex semi-simple algebraic group and $H$ is a parabolic subgroup. Let $K$ be a maximal compact subgroup of $G.$ Then $K$ acts transitively on $X.$ We assume moreover that $\omega$ is fixed by action of $K.$ In this case we can regularize singular $(\omega,m)$-subharmonic functions by using the group action which preserves the metric. 

Let $\f$ be a continuous $(\omega,m)$-subharmonic function on $X.$ 
We consider the following regularizing sequence 
$$
\f_{\epsilon}(x):=\int_K \f(g^{-1}.x)\chi_{\epsilon}(g)dg,
$$
where $dg$ is the Haar measure on $K$ and $\chi_{\epsilon}$ are cut-off functions whose supports decreases to $\{e\}$ (the identity of $K$), and $\int_K \chi_{\epsilon}(g)dg=1, \forall \epsilon>0.$  \\
It follows from \cite{G99}, \cite{Hu94} that $\f_{\epsilon}$ is smooth for every $\epsilon>0.$ 
\begin{theorem}\label{thm: reg}
Let $\f$ be a continuous $(\omega,m)$-subharmonic function on $X.$ Then for each $\epsilon>0$, $\f_{\epsilon}$ is   smooth $(\omega,m)$-subharmonic  and 
$$
\lim_{\epsilon\to 0}\f_{\epsilon}=\f
$$
uniformly on $X.$ 
\end{theorem}
\begin{proof}
The uniform convergence always holds for continuous functions.
Let us show the second assertion. Let $\alpha_1,...,\alpha_{m-1}$ be $(\omega,m)$-positive closed (1,1)-forms on $X$, and denote (for short) $\alpha=\alpha_1\wedge...\wedge \alpha_{m-1}.$ Let $\Lc_g$ denote the left action of $g\in K$, i.e. 
$$
\Lc_g(x)=g.x, \ \ x\in X.
$$
Then $\Lct_g\alpha_j$ is also $(\omega,m)$-positive for every $j,$ since $\Lct_{g^{-1}} \omega=\omega,$ and
$$
\Lct_{g^{-1}} (\Lct_g\alpha ^k \wedge \omega^{n-k})= \alpha^k\wedge \omega^{n-k}.
$$  
Fix a positive test function $\psi.$ We  have
\begin{eqnarray*}
&&\int_X \psi (\omega+dd^c \f_{\epsilon})\wedge \alpha\wedge \omega^{n-m}= \int_X \psi \alpha\wedge\omega^{n-m+1}+ \int_X \f_{\epsilon} dd^c \psi\wedge \alpha\wedge  \omega^{n-m}\\
&&\, \,\,\,\, \,\,\,=\int_X \psi \alpha\wedge\omega^{n-m+1}+ \int_X\Big(\int_K\Lct_{g} \f \chi_{\epsilon}(g) dg\Big)dd^c \psi\wedge \alpha\wedge  \omega^{n-m}\\
&&\, \,\,\,\, \,\,\,=\int_X \psi \alpha\wedge\omega^{n-m+1}+ \int_K\Big(\int_X\Lct_{g} \f  dd^c \psi\wedge \alpha\wedge  \omega^{n-m}\Big)\chi_{\epsilon}(g)dg\\
&&\, \,\,\,\, \,\,\,=\int_K\Big(\int_X\psi (\omega+dd^c \Lct_g \f)  \wedge\alpha\wedge\omega^{n-m}\Big)\chi_{\epsilon}(g)dg\\
&&\, \,\,\,\, \,\,\,=\int_K\Big(\int_X \psi (\omega+\Lct_g dd^c \f)\wedge \alpha\wedge  \omega^{n-m}\Big)\chi_{\epsilon}(g)dg\\
&&\, \,\,\,\, \,\,\,=\int_K\Big(\int_X \psi \Lct_g\big[(\omega+dd^c\f)\wedge\Lct_{g^{-1}}\alpha\wedge\omega^{n-m}\big]\Big)\chi_{\epsilon}(g)dg\geq 0.
\end{eqnarray*}
\end{proof}
\begin{rmk}
Thanks to Theorem \ref{thm: reg}, every continuous $(\omega,m)$-subharmonic function belongs to $\Pm(X,\omega).$
\end{rmk}
\noindent {\bf Proof of Theorem B.} Let $\f$ be the unique continuous solution to (\ref{eq: He}). For $h\in K$, let $\f_h(x):=\f(h.x), \ x\in X.$ If $u$ is smooth then 
$$
\Vert u_h-u\Vert^2_{L^2} \leq Cdist^2(h,e)\int_X (-u)dd^c u\wedge \omega^{n-1},
$$
where $C$ is some universal constant. Then, it follows from the approximation theorem (Theorem \ref{thm: reg}) that 
$$
\Vert \f_h-\f\Vert_{L^2(X)}\leq Cdist(h,e).
$$
For fixed $h\in K$, observe that $\f_h$ is $(\omega,m)$-subharmonic and satisfies
$$
H_m(\f_h)=F(h.x,\f(h.x))\omega^n.
$$
Thus, by applying Theorem \ref{thm: Lr-Linfty} with $r=2$ we obtain
$$
\Vert \f_h-\f\Vert_{L^{\infty}} \leq C'.\Vert \f_h-\f\Vert_{L^2(X)}^{\gamma},
$$
where $0<\gamma<\frac{2(mp-n)}{mnp+2mp-2n}$ is a given constant and $C'>0$ is another constant which does not depend on $h.$ We thus get
$$
\Vert \f_h-\f\Vert_{L^{\infty}(X)}\leq C.C'dist (h,e)^{\gamma}, \ \forall h\in K.
$$ 
This yields the $\gamma$-H\"{o}lder continuity of $\f$ (see \cite{EGZ09}).

 \vskip .2cm

Lu Hoang Chinh

Laboratoire Emile Picard

UMR 5580, Universit\'e Paul Sabatier

118 route de Narbonne

31062 TOULOUSE Cedex 04 (FRANCE)

lu@math.univ-toulouse.fr  
  \end{document}